\RequirePackage[l2tabu, orthodox]{nag}
\PassOptionsToPackage{table,xcdraw}{xcolor}

\newif\ifllncs
\newif\ifspringer
\newif\ifnonspringer

\llncsfalse
\springerfalse
\nonspringertrue

\makeatletter
\providecommand*{\input@path}{}
\g@addto@macro\input@path{{llncs/}{include/}{../include/}}
\makeatother

\ifllncs
    \documentclass[oribibl,envcountsame]{llncs}
    \usepackage[misc]{ifsym} %
    \author{
        Aleksandr M. Kazachkov\inst{1}\textsuperscript{(\Letter)}\orcidID{0000-0002-4949-9565} 
        \and 
        Egon Balas\inst{2}%
        \thanks{E.~Balas passed away during the preparation of this manuscript, which started when both authors were at Carnegie Mellon University. The core ideas and early results are documented in the PhD dissertation of \citet[Chapter~5]{Kazachkov18}. A.M.~Kazachkov completed the computational experiments, analysis, and writing independently.}
    }
    \institute{
        University of Florida, Gainesville, FL, USA \\
        \email{akazachkov@ufl.edu}
        \and
        Carnegie Mellon University, Pittsburgh, PA, USA \\
        \email{eb17@andrew.cmu.edu}
    }
    
	\newenvironment{proof*}[1]
  		{\renewcommand{\proofname}{#1}%
  			\begin{proof}}
  			{\end{proof}}
\fi

\ifspringer
    \RequirePackage{amsmath}
    \documentclass[smallextended,envcountsame,nospthms]{springer/svjour3}
    \journalname{Mathematical Programming}
    \smartqed  %
    \usepackage[misc]{ifsym}
    
    \author{%
        Aleksandr M. Kazachkov\inst{1}\textsuperscript{(\Letter)}\orcidID{0000-0002-4949-9565} 
        \and 
        Egon Balas\inst{2}
    }
    
    \institute{%
             \textsuperscript{\href{mailto:akazachkov@ufl.edu}{\Letter}} A. M. Kazachkov
             \at
             University of Florida \\
             \email{akazachkov@ufl.edu} \\
             ORCID: 0000-0002-4949-9565
             \and
             E. Balas
             \at
             Carnegie Mellon University \\
            \email{eb17@andrew.cmu.edu}
    }
    \date{\today}
    \authorrunning{Kazachkov and Balas}
\fi

\ifnonspringer
    \documentclass[12pt]{article}
    \usepackage[margin=1in]{geometry}
	
	\newcommand*{\email}[1]{%
    		\href{mailto:#1}{#1}\par
    		}
    
	\author{%
		Aleksandr M. Kazachkov%
			\thanks{University of Florida, Gainesville, FL, USA (\email{akazachkov@ufl.edu}).}
		\and
		Egon Balas%
			\thanks{E.~Balas passed away during the preparation of this manuscript, which started when both authors were at Carnegie Mellon University. The core ideas and early results are documented in the PhD dissertation of \citet[Chapter~5]{Kazachkov18}. A.M.~Kazachkov completed the computational experiments, analysis, and writing independently.}
	}
	
	\date{February 28, 2023}

\fi

\title{Monoidal Strengthening of Simple \texorpdfstring{$\mathcal{V}$}{V}-Polyhedral Disjunctive Cuts}

\usepackage{akazachk}

\ifnonspringer
    \newtheorem{theorem}{Theorem}
    \newtheorem*{theorem*}{Theorem}

    \newtheorem{lemma}[theorem]{Lemma}

    \theoremstyle{remark}
    
    \newtheorem*{subclaim*}{Claim}

    \theoremstyle{definition}

    \newtheorem*{definition*}{Definition}
    \newtheorem{example}{Example}
    \newtheorem*{example*}{Example}
    
\fi

\renewcommand{\vec}[1]{\boldsymbol{#1}}

\newcommand{\rlaptext}[1]{\rlap{\text{#1}}}
\DeclareMathAlphabet{\mathpzc}{OT1}{pzc}{m}{it}

\newcommand{\pointset}{\mathcal{P}}
\newcommand{\rayset}{\mathcal{R}}

\newcommand{\intvars}{\mathcal{I}}
\newcommand{\PI}{P_{I}}
\newcommand{\PD}{P_{D}}
\newcommand{\disjTermsIndexSet}{\mathcal{T}}
\newcommand{\M}{\ref{M}}
\newcommand{\alphatk}{\ref{alphatk}}
\newcommand{\Deltat}{\Delta^t}

\newcommand{\monoidSet}{\mathbbm{M}}
\newcommand{\monoid}{\vec{m}}

\newcommand{\numRowsP}{q}
\newcommand{\numRowsDt}{q_t}
\newcommand{\numRowsPt}{q_t'}
\newcommand{\numRowsQt}{\hat{q}_t}

\newcommand{\mxrow}[2]{#1_{#2 \boldsymbol{\cdot}}}
\newcommand{\mxcol}[2]{#1_{\boldsymbol{\cdot} #2}}

\newcommand{\submx}[2]{#1_{#2}}

\newcommand{\goodVPCSet}{$\ge$10\%}

\newcommand{\instance}[1]{\texttt{#1}}

\usepackage{relsize}
\def\ifmonospace{\ifdim\fontdimen3\font=0pt }

\def\Cpp{%
\ifmonospace%
    C++%
\else%
    C\nolinebreak[4]\raisebox{0.5ex}{\tiny\textbf{++}}%
\fi%
\spacefactor1000 }

\newcommand{\exqed}{\hfill$\blacksquare$}

\begin{document}
\maketitle

\begin{abstract}
Disjunctive cutting planes can tighten a relaxation of a mixed-integer linear program. Traditionally, such cuts are obtained by solving a higher-dimensional linear program, whose additional variables cause the procedure to be computationally prohibitive. Adopting a $\mathcal{V}$-polyhedral perspective is a practical alternative that enables the separation of disjunctive cuts via a linear program with only as many variables as the original problem. The drawback is that the classical approach of monoidal strengthening cannot be directly employed without the values of the extra variables appearing in the extended formulation. We derive how to compute these values from a solution to the linear program generating $\mathcal{V}$-polyhedral disjunctive cuts. We then present computational experiments with monoidal strengthening of cuts from disjunctions with as many as 64 terms. Some instances are dramatically impacted, with strengthening increasing the gap closed by the cuts from 0 to 100\%. However, for larger disjunctions, monoidal strengthening appears to be less effective, for which we identify a potential cause.
\ifspringer
  \keywords{Mixed-integer linear programming \and Cutting planes \and Disjunctive programming \and Monoidal strengthening}
\fi
\end{abstract}

\section{Introduction}
\label{sec:intro}

Disjunction-based cutting planes, or \emph{disjunctive cuts}, are a strong class of valid inequalities for mixed-integer programming problems,
which can be used as a framework for analyzing or generating general-purpose cuts~\cite{Balas79}.
Their strength comes at a high computational cost, due to which only very special cases of disjunctive cuts have been deployed in optimization solvers.
As a step towards practicality, \citet{BalKaz22+_vpc-arxiv} introduce a relaxation-based $\mathcal{V}$-polyhedral paradigm for disjunctive cuts,
which trades off some theoretical strength for computational efficiency.
The approach selects a small number of points and rays whose convex hull forms a relaxation of the disjunction;
as a result, some potential cuts are no longer valid,
but strong cuts are nevertheless guaranteed to be obtainable.
Further, cuts from this relaxation, called \emph{$\mathcal{V}$-polyhedral (disjunctive) cuts} (VPCs), can be generated via a relatively compact linear program,
called the \emph{point-ray linear program} (PRLP),
compared to the usual higher-dimensional \emph{cut-generating linear program} (CGLP) for disjunctive cuts~\cite{Balas79,BalCerCor93,BalCerCor96}.
Hence, with VPCs, it is more computationally efficient to improve the disjunction by adding terms and increase the relaxation quality, 
thereby accessing disjunctive cuts that differ substantially from the families of cuts typically applied in solvers.

VPCs improve the average \emph{(integrality) gap closed} substantially relative to \emph{Gomory mixed-integer cuts} (GMICs) and other standard cuts in solvers.
However, the computational experiments by \citet{BalKaz22+_vpc-arxiv} reveal a curiosity: there are instances for which GMICs (which can be derived as cuts from a two-term disjunction) remain stronger than VPCs even when using large variable disjunctions.
For example, for the instance \instance{10teams}, originally part of the 3rd Mixed Integer Programming Library (MIPLIB)~\cite{MIPLIB3}, GMICs close 100\% of the integrality gap, while VPCs from a 64-term disjunction close 0\% of the gap.

A potential explanation for this phenomenon is that GMICs benefit from a \emph{strengthening} procedure that cannot be directly applied to VPCs.
Specifically, the GMIC two-term disjunction can be obtained via \emph{monoidal strengthening} of a disjunction on a single variable~\cite{BalJer80, BalPer03, KazSer23}.
Monoidal strengthening of cuts from more general disjunctions is also possible, but the procedure ostensibly requires a \emph{simple} disjunction, where each term only imposes a single new constraint.
This is not a theoretical barrier, as any cut from a general disjunction can also be derived from a simple disjunction obtained from the general one by aggregating the constraints defining each disjunctive term.
The multipliers for this aggregation are precisely the Farkas certificate for the validity of the cut.
The key challenge for VPCs is that this certificate is not readily available, because the PRLP only has variables for the cut coefficients,
compared to the CGLP that explicitly includes variables for the Farkas multipliers.
Our contributions, summarized next, are to identify a way to efficiently apply monoidal strengthening for the particular version of the VPC framework introduced in \citet{BalKaz22+_vpc-arxiv}, as well as to implement and computationally evaluate this strengthening idea.

\paragraph{Contributions.}
Given a VPC, one can solve the CGLP with cut coefficients fixed and retrieve the required values of the aggregation multipliers, in order to apply monoidal strengthening.
Unfortunately, the computational effort associated to this is likely to be prohibitive.
Our first contribution, discussed in \cref{sec:correspondence}, is observing that solving the CGLP is unnecessary: it suffices to use the inverse of an easily-identified nonsingular matrix per disjunctive term.
Furthermore, for the type of \emph{simple} VPCs proposed and tested by \citet{BalKaz22+_vpc-arxiv}, this inverse is readily available within the cut generation process.

Next, in \cref{sec:computation}, we discuss computational experiments with strengthening simple VPCs on a set of benchmark instances.
We compare the strength to unstrengthened VPCs and to GMICs, for disjunctions ranging in size up to 64 terms.
We find that strengthening can significantly improve the gap closed for some instances.
Furthermore, we see that GMICs and unstrengthened VPCs tend to be complementary in terms of which instances they benefit,
but applying monoidal strengthening enables the two families to be simultaneously effective for more instances.
The results are most striking for two-term disjunctions, in which strengthened VPCs close 40\% more gap than unstrengthened VPCs, on average.
For example, returning to the instance \instance{10teams}, the VPCs from a single variable disjunction close 0\% of the integrality gap, but this value goes to 100\% after strengthening the cuts.
However, as the size of the disjunction increases, the relative improvement by strengthening becomes smaller.
Our final contribution, in \cref{sec:example:large-disjunction}, is identifying a theoretical source of this weakness.

\paragraph{Related Work.}

A focal point in the literature on monoidal strengthening for disjunctive cuts~\cite{BalJer80} (see also \citet[Section~7]{Balas79})
is the special case of \emph{split disjunctions},
which are \emph{parallel} two-term disjunctions that are used for GMICs and related cut families.
In this context, the use of the CGLP leads to \emph{lift-and-project cuts} (L\&PCs)~\cite{BalCerCor93},
to which monoidal strengthening can be applied~\cite[Section~2.4]{BalCerCor96}.
The family of strengthened L\&PCs is equivalent to GMICs,
as shown by \citet{BalPer03},
and to mixed-integer rounding inequalities~\cite{NemWol88,NemWol90}, as discussed in \citet{CorLi01}.
\citet{BalPer03} provide an appealing geometric interpretation of this connection via intersection cuts~\cite{Balas71}:
every undominated L\&PC can be derived as an intersection cut from a basis in the original problem space.
As a result, L\&PCs can be generated without explicitly building the CGLP
and without hindering \latin{a posteriori} strengthening of the cuts.
\citet{Bonami12} presents a different method for separating L\&PCs in the original space of variables that is also amenable to strengthening.
Avoiding formulating the higher-dimensional CGLP is the key advance that has enabled the effective inclusion of L\&PCs in several solvers.

Sidestepping the CGLP continues to be crucial to move beyond split disjunctions.
However, the aforementioned approaches~\cite{BalPer03,Bonami12} rely on properties of the split set;
for example, with general disjunctions, there exist cuts that dominate all intersection cuts~\cite{AndCorLi05,Kis14,BalKis16},
so one cannot hope to merely pivot among bases in the original space.
Nonetheless, a stream of work~\cite{JudSheRibFau06_complementarity,BonConCorMolZam13,Kis14} 
extends cut generation in the original space to general two-term disjunctions,
and monoidal strengthening applies to the resulting cuts~\cite{FisPfe17}.
No further extension of this technique to more general disjunctions has been reported in the literature.

This motivates the use of VPCs, due to the PRLP's advantage of having the same number of variables as the original problem.
The difficulty is that a description of a polyhedron using points and rays may be exponentially larger than using inequalities, causing exponentially many constraints in the PRLP.
This naturally leads to row generation in prior work by \citet{PerBal01} and \citet{LouPoiSal15} when invoking the $\mathcal{V}$-polyhedral perspective.
In the experiments by \citet{PerBal01}, for disjunctions with 16 terms, separating cuts via the PRLP with row generation is an order of magnitude faster than via the CGLP.
Nonetheless, row generation is time consuming, as multiple PRLPs must be solved to find one valid inequality.

The remedy by \citet{BalKaz22+_vpc-arxiv} is to construct a relaxation of each disjunctive term, where the resulting PRLP has few rows and immediately produces valid cuts.
This is successful at quickly generating cuts from large disjunctions,
but the average gap closed by the cuts alone is less than that from GMICs.
It is only when VPCs and GMICs are used together that a marked improvement in gap closed is observed,
which shows that VPCs affect a different region of the relaxation than GMICs.
However, as mentioned with the \instance{10teams} instance in which GMICs close all of the gap, while VPCs close none,
the results also suggest that the absence of strengthening for VPCs is a significant deficiency.

As discussed, the vanilla monoidal strengthening presented by \citet{BalJer80} does not directly apply to VPCs due to the lack of the values of the aggregation multipliers.
\citet[Section~6]{BalQua13} show that a cross-polytope disjunction, arising from using multiple rows of the simplex tableau, can be strengthened by \emph{modularizing} the inequalities defining the disjunction, replacing the coefficients of integer-restricted nonbasic variables,
and they prove the form of the optimal strengthening for the two-row case.

An alternative to monoidal strengthening is the group-theoretic approach~\cite{Johnson74,GomJoh72a},
equivalent to monoidal strengthening under some conditions.
Specifically, ``trivial lifting'' has been applied to simple disjunctions~\cite{Espinoza10,DeyWol10,BasBonCorMar11a,DeyLodTraWol14,XavFukPoi21}.
Evaluating the trivial lifting is expensive in general~\cite{FukPoiXav19}, and it does not directly apply to arbitrary disjunctive cuts.

While this paper exclusively approaches disjunctive cut generation via the VPC framework,
there exist other methods for producing strong disjunctive cuts without solving the higher-dimensional CGLP.
Any such approach could potentially benefit from the efficient computation of a Farkas certificate.
For example, a common technique in the literature is to use a disjunction to strengthen cuts via tilting, which has been applied to linear and nonlinear integer optimization problems~\cite{Perregaard03,KilLinLueMil14,Kazachkov18,KroMis21_disj-cut-strengthening-convex-MINLP}.

\section{Notation and Background}
\label{sec:background}

Our target is to find strong valid cuts to tighten the natural linear relaxation of the mixed-integer linear program below, given rational data:
    \begin{equation}
        \begin{aligned}
          \min_{x \in \R^n} \quad 
            &c^\T x
            \\
            & \mxrow{A}{i} x \ge b_i &\quad& \text{for $i \in [\numRowsP]$,}
            \\
            & x_j \ge 0 &&\text{for $j \in [n]$,}
            \\
            & x_j \in \Z &&\text{for $j \in \intvars$} .
        \end{aligned}
        \label{IP}\tag{IP}
    \end{equation}
Here, $[n] \defeq \{1,\ldots,n\}$ for any integer $n$,and $\intvars \subseteq [n]$ is the set of integer-restricted variables.
For a given matrix $A$, we denote the $i$th row by ``$\mxrow{A}{i}$'' and the $j$th column by ``$\mxcol{A}{j}$''.
Let $\PI$ denote the feasible region of \cref{IP}, and let $P \defeq \{x \in \nonnegreals^n \suchthat Ax \ge b\}$.

One way to strengthen the formulation $P$ (with respect to $\PI$) is to use logical conditions to formulate a \emph{disjunction}, from which valid inequalities for $\PI$ can then be derived.
Suppose $\vee_{t \in \disjTermsIndexSet} (D^t x \ge D^t_0)$ is a valid disjunction, in the sense that $\PI \subseteq \cup_{t \in \disjTermsIndexSet} \{x \in \R^n \suchthat D^t x \ge D^t_0\}$.
Let 
    $Q^{t} \defeq \{x \in P \suchthat D^t x \ge D^t_0\}$.
This is an \emph{$\mathcal{H}$-polyhedral (inequality) description}.
We assume $Q^t \ne \emptyset$ for all $t \in \disjTermsIndexSet$.

Let $P^t \defeq \{x \in \R^n \suchthat A^t x \ge b^t\}$ denote a relaxation of $Q^t$, 
where $A^t x \ge b^t$ is defined by a subset of the constraints defining $Q^t$.
For the VPC procedure, we must ensure that $P^t$ has relatively few extreme points and rays, \latin{i.e.}, it has a compact \emph{$\mathcal{V}$-polyhedral description} $(\pointset^t, \rayset^t)$, so that $P^t = \conv(\pointset^t) + \cone(\rayset^t)$.
Define the \emph{disjunctive hull} 
    $
        \PD
        \defeq
        \cl\conv(\cup_{t \in \disjTermsIndexSet} P^t)
        ,
    $
    which can be described by the point-ray collection
  $(\pointset, \rayset) \defeq (\cup_{t \in \disjTermsIndexSet} \pointset^t, \cup_{t \in \disjTermsIndexSet} \rayset^t)$.
For $t \in \disjTermsIndexSet$, let $\numRowsPt$ be the number of rows of $A^t$.
We first summarize some important disjunctive programming concepts and the two cut-generating paradigms that we are relating.

\paragraph{CGLP.}
One way to generate valid cuts for $\PD$ is through the CGLP, which is an application of \emph{disjunctive programming duality}~\cite[Section~4]{Balas79}.
Specifically, an inequality $\alpha^\T x \ge \beta$ is valid for $\PD$ if and only if the inequality is valid for each $P^t$, $t \in \disjTermsIndexSet$.
Consequently, by Farkas's lemma~\cite{Farkas02}, 
$\alpha^\T x \ge \beta$ is valid for $\PD$ if and only if the following system is feasible, in variables $(\alpha,\beta,\{v^t\}_{t \in \disjTermsIndexSet})$, where $v^t \in \R^{1 \times \numRowsPt}$ is a row vector of appropriate length for each $t \in \disjTermsIndexSet$:
\begin{align}
  \label{gen-CGLP}
  \left.
  \begin{array}{l}
    \alpha^\T = v^t A^t \\
    \beta \le v^t b^t \\
    v^t \in \nonnegreals^{\numRowsPt}
  \end{array}
  \ \ \right\}
  \ \ 
  \rlaptext{for all $t \in \disjTermsIndexSet$.}
\end{align}
We refer to $\{v^t\}_{t \in \disjTermsIndexSet}$ as the \emph{Farkas certificate} for the validity of $\alpha^\T x \ge \beta$ for $\PD$.

To generate cuts with \eqref{gen-CGLP},
one typically maximizes the violation with respect to a $\PI$-infeasible point, 
after adding a normalization, which can be a crucial choice~\cite{FisLodTra11}.
For example, the constant of the cut can be fixed to $\bar{\beta} \in \reals$:
  \begin{equation}
  \label{CGLP-fixed-beta}
  \tag{CGLP($\bar{\beta}$)}
    \left\{(
        \alpha,\{v^t\}_{t \in \disjTermsIndexSet}) 
        \suchthat 
        (\alpha,\bar{\beta},\{v^t\}_{t \in \disjTermsIndexSet}) 
        \text{ is feasible to \eqref{gen-CGLP}}
    \right\}.
  \end{equation}

\paragraph{PRLP.}
An alternative way to generate disjunctive cuts is through the \emph{reverse polar} of $\PD$~\cite[Section~5]{Balas79}, which is defined with respect to a given $\bar{\beta} \in \R$ as 
  \begin{equation*}
    \left\{ 
        \alpha \in \R^n \suchthat \alpha^\T x \ge \bar{\beta} \text{ for all $x \in \PD$} 
    \right\}.
  \end{equation*}
Clearly this captures all of the valid inequalities for $\PD$ whose constant is equal to $\bar{\beta}$.
Since $x \in \PD$ if and only if $x \in \conv(\pointset) + \cone(\rayset)$, it holds that $\alpha^\T x \ge \bar{\beta}$ is valid for $\PD$ if and only if it is satisfied by all of the points and rays in $(\pointset,\rayset)$.
This yields the system \eqref{PRLP-fixed-beta}, in variables $\alpha \in \R^n$, for a fixed $\bar{\beta}$:
  \begin{align}
  \label{PRLP-fixed-beta}
  \tag{PRLP($\bar{\beta}$)}
  \begin{aligned}
    \alpha^\T p &\ge \bar{\beta} &\quad& \rlaptext{for all $p \in \pointset$} \\
    \alpha^\T r &\ge 0 &\quad& \rlaptext{for all $r \in \rayset$.}
  \end{aligned}
  \end{align}
The feasible solutions to \cref{PRLP-fixed-beta} are what we refer to as VPCs.
  
As discussed, the advantage of \eqref{PRLP-fixed-beta} over \eqref{CGLP-fixed-beta} is the absence of the Farkas multipliers as variables, so VPCs are generated without requiring a lifted space.
As we see next, the disadvantage to \cref{PRLP-fixed-beta} is that these missing variables are used in strengthening the cuts after they are generated.

\paragraph{Monoidal strengthening.}

\citet{BalJer80} strengthen cuts with a \emph{monoid}:
  \begin{equation}
  \label{M}
  \tag{\ensuremath{\monoidSet}}
    \monoidSet \defeq \left\{ \monoid \in \Z^{\card{\disjTermsIndexSet}} \suchthat \sum_{t \in \disjTermsIndexSet} \monoid_t \ge 0 \right\}.
  \end{equation}
It is also assumed that, 
for each $t \in \disjTermsIndexSet$,
there exists a finite lower bound vector $\ell^t$ such that $D^t x \ge \ell^t$ for all $x \in \PI$.
Let $\Deltat \defeq D^t_0 - \ell^t$.

To strengthen the cut, we improve the underlying disjunction.
Specifically, given a valid disjunction $\vee_{t \in \disjTermsIndexSet} (D^t x \ge D^t_0)$,
for any $\monoid \in \ref{M}$ and $k \in \intvars$,
the disjunction $\vee_{t \in \disjTermsIndexSet} (\tilde{D}^t x \ge \tilde{D}^t_0)$ is also valid,
where $\mxcol{\tilde{D}^t}{k} \defeq \mxcol{D^t}{k} + \Deltat \monoid_t$, 
and $\mxcol{\tilde{D}^t}{j} = \mxcol{D^t}{j}$ for all $j \ne k$.
The strengthened cut is obtained by applying the Farkas certificate of the unstrengthened cut to the strengthened disjunction.

Let $\numRowsDt$ denote the number of constraints in $D^t x \ge D^t_0$ for term $t \in \disjTermsIndexSet$.
Given row vectors
	$(u^t,u^t_0) \in \nonnegreals^{1 \times \numRowsP} \times \nonnegreals^{1 \times \numRowsDt}$,
define
	\begin{equation}
	\label{alphatk}
	\tag{\ensuremath{\alpha^t_k}}
	  \alpha^t_k \defeq {u}^t \mxcol{A}{k} + u^t_0 \mxcol{D^t}{k}.
	\end{equation}
Then (using an appropriate CGLP) the cut $\alpha^\T x \ge \beta$ is valid for $\PD$,
where
    \begin{equation*}
            \alpha_k \defeq \max_{t \in \disjTermsIndexSet} \{ \alphatk \}
            \qquad
            \text{ and }
            \qquad
            \beta \defeq \min_{t \in \disjTermsIndexSet} \{ u^t b + u^t_0 D^t_0 \}.
    \end{equation*}
(The above applies to cuts valid for $\vee_{t \in \disjTermsIndexSet} Q^t$; for $\PD$, assume a value of zero for the multipliers on constraints of $Q^t$ that are not present in $P^t$.)
Define $\hat{u}^t_k \defeq \alpha_k - \alphatk$.
We now apply monoidal strengthening to the cut $\alpha^\T x \ge \beta$.
	
\begin{theorem}[{\cite[Theorem~3]{BalJer80}}]
\label{thm:monoidal-strengthening-structural-space}
  Given 
    $({u}^t,u^t_0) \in \nonnegreals^{1 \times \numRowsP} \times \nonnegreals^{1 \times \numRowsDt}$ for $t \in \disjTermsIndexSet$,
  the inequality $\widetilde{\alpha}^\T x \ge \beta$ is valid for $\PI$,
  where
  $\widetilde{\alpha}_k \defeq \alpha_k$ for $k \notin \intvars$,
  and,
  for $k \in \intvars$,
    \begin{equation*}
      \widetilde{\alpha}_k 
        \defeq \inf_{\monoid \in \M} \max_{t \in \disjTermsIndexSet} \left\{ \alphatk + u^t_0 \Deltat \monoid_t \right\}
        = \alpha_k + \inf_{\monoid \in \M} \max_{t \in \disjTermsIndexSet} \left\{ -\hat{u}^t_k + u^t_0 \Deltat \monoid_t \right\}.
    \end{equation*}
\end{theorem}

Thus, the Farkas certificate 
    $\{(u^t, u^t_0)\}_{t \in \disjTermsIndexSet}$ 
is used for monoidal strengthening.
Computing these values without solving the CGLP is our next target.

\section{Correspondence Between PRLP and CGLP Solutions}
\label{sec:correspondence}

Let $\bar{\alpha}^\T x \ge \bar{\beta}$ be a valid inequality for $\PD$, corresponding to a feasible solution to \cref{PRLP-fixed-beta}.
Our goal is to compute Farkas multipliers certifying the cut's validity without explicitly solving the CGLP.
While one can solve for values $v^t$ that satisfy
    $
            \bar{\alpha}^\T = v^t A^t,\,
            \bar{\beta} = v^t b^t,\,
            v^t \ge 0
        ,
    $
we provide an improvement via basic linear programming concepts.
We first present a special case in \cref{sec:simple-vpcs,sec:simple}, when the disjunctive terms $P^t$ are not primal degenerate,
a condition that is satisfied by the VPC procedure implemented for our experiments.
Then, \cref{sec:general-vpcs} discusses a challenge posed by the general case.

We assume that $\bar{\alpha}^\T x \ge \bar{\beta}$ is supporting for all terms in $\disjTermsIndexSet$.
This is for ease of notation, as otherwise we would need to add an index $t$ to the constant side.
Concretely, the assumption is without loss of generality because, for any term $t \in \disjTermsIndexSet$, we can increase the constant side of the cut until we obtain an inequality $\bar{\alpha}^\T x \ge \bar{\beta}_t$ that is supporting for term $t$, though perhaps invalid for other terms.
The value of $\bar{\beta}_t$ can be quickly calculated by taking the dot product of $\bar{\alpha}$ with every point in $\pointset^t$.
We can then find a certificate $v^t$ of the validity of $\bar{\alpha}^\T x \ge \bar{\beta}_t$ for $P^t$,
which also serves as a certificate for the weaker inequality $\bar{\alpha}^\T x \ge \bar{\beta}$.
We state, without proof, a slightly more general version of this in \cref{lem:weaker-certificate}.

\begin{lemma}
\label{lem:weaker-certificate}
    For $t \in \disjTermsIndexSet$, let $C^t \supseteq P^t$ and $\bar{\beta}_t \ge \bar{\beta}$ such that $\bar{\alpha}^\T x \ge \bar{\beta}_t$ is valid for $C^t$.
    Then, given any Farkas certificate for the validity of the inequality $\bar{\alpha}^\T x \ge \bar{\beta}_t$ for $C^t$,
    the same multipliers certify that $\bar{\alpha}^\T x \ge \bar{\beta}$ is valid for $P^t$. \qed
\end{lemma}

For convenience, we introduce extra notation to refer to the feasible region of $Q^t$ as $\hat{A}^t x \ge \hat{b}^t$,
and we define the number of these constraints as $\numRowsQt \defeq \numRowsP + \numRowsDt + n$.
For $N \subseteq [\numRowsQt]$, define $\submx{\hat{A}^t}{N} x \ge \submx{\hat{b}^t}{N}$ as the constraints of $Q^t$ indexed by $N$.

\subsection{Simple VPCs}
\label{sec:simple-vpcs}

Our experimental setup in \cref{sec:computation} follows that of \citet{BalKaz22+_vpc-arxiv}, who focus on a variant of the VPC framework called \emph{simple} VPCs.
Let $p^t$ be a vertex of $Q^t$, for $t \in \disjTermsIndexSet$.
There exists a \emph{cobasis} for $p^t$, a set of $n$ linearly independent constraints among those defining $Q^t$ that are tight at $p^t$.
Let $N^t \subseteq [\numRowsQt]$ denote the indices of these $n$ constraints,
and define the \emph{basis cone} 
    $
        C^t \defeq 
            \{
                x \in \R^n \suchthat
                \submx{\hat{A}^t}{N^t} x \ge \submx{\hat{b}^t}{N^t}
            \}.
    $
The inequality $\bar{\alpha}^\T x \ge \bar{\beta}$ is a simple VPC if $P^t$ is a basis cone for each term.
The (translated) cone $C^t$ has a particularly easy $\mathcal{V}$-polyhedral representation:
there is a single extreme point $p^t$, and there are $n$ extreme rays $\{r^i\}_{i \in [n]}$.
The $i$th extreme ray of $C^t$ corresponds to increasing the ``slack'' on the $i$th constraint defining $C^t$~\cite[Chapter~6]{ConCorZam14}.
\cref{lem:A^t-invertible} states that, 
for simple VPCs,
the values of the variables $\{v^t\}_{t \in \disjTermsIndexSet}$ to \eqref{CGLP-fixed-beta} can be computed via the dot product of the cut coefficients with the rays of $C^t$.

\begin{lemma}
\label{lem:A^t-invertible}
    Let $C^t$ be a basis cone defined by $N^t$, the indices of $n$ linearly independent constraints of $Q^t$.
    If $\bar{\alpha}^\T x \ge \bar{\beta}$ is valid for 
    $C^t$,
    then the multiplier on constraint $i \in [n]$ of $C^t$ has value
        $
            v^t_i = \bar{\alpha}^\T r^i,
        $
    where $r^i$ is column $i$ of $(\submx{\hat{A}^t}{N^t})^{-1}$.
\end{lemma}
\begin{proof}
    Add nonnegative slack variables $s^t_{N^t}$ for each row indexed by $N^t$, so that $\submx{\hat{A}^t}{N^t} x - s^t_{N^t} = b^t_{N^t}$.
    Then observe that, being a cobasis, $\submx{\hat{A}^t}{N^t}$ is invertible,
    so 
        $
            x 
                = (\submx{\hat{A}^t}{N^t})^{-1} b^t_{N^t} + (\submx{\hat{A}^t}{N^t})^{-1} s^t_{N^t} 
                = p^t + \sum_{i \in N^t} r^i s^t_i.
        $
    The last equality follows from the derivation of the rays of $C^t$; see, for example, \citet[Chapter~6]{ConCorZam14}.
\end{proof}

Therefore, for simple VPCs, the Farkas certificate can be computed with no extra effort when given the point-ray representation of $\PD$.
Moreover, \citet{BalKaz22+_vpc-arxiv} obtain simple VPCs from the leaf nodes of a partial branch-and-bound tree and use $p^t$ as the optimal solution to the linear relaxation at each leaf;
implemented carefully, this can further reduce the computational load for generating then strengthening VPCs, as the values of the rays can be read from the optimal tableau,
which is typically readily available from a solver.

\subsection{Relaxations Without Primal Degeneracy}
\label{sec:simple}

Suppose the relaxation $P^t \supseteq Q^t$ is a \emph{simple polyhedron}, in which every extreme point and ray is defined by a \emph{unique} basis~\cite{Ziegler95}.
The basis cone $C^t$ used for simple VPCs is one example.
While the basis cone setting may seem quite narrow, it turns out to encompass more general situations.
Specifically, there always exists a basis cone $C^t \supseteq P^t$ such that $\bar{\alpha}^\T x \ge \bar{\beta}$ is valid and supporting for $C^t$.

\begin{lemma}
\label{lem:nondegenerate-case}
    Let $P^t$ be a simple polyhedron,
    and suppose the point-ray collection $(\pointset^t,\rayset^t)$ satisfies $P^t = \conv(\pointset^t) + \cone(\rayset^t)$.
    Let $\bar{\alpha}^\T x \ge \bar{\beta}$ be a valid inequality for $P^t$.
    Then there exists a vertex $p^t \in \pointset^t$ such that $\bar{\alpha}^\T x \ge \bar{\beta}$ is valid for the basis cone $C^t$ associated to $p^t$, defined with respect to the constraints of $P^t$.
\end{lemma}
\begin{proof}
    Let $p^t$ be an optimal solution to 
        $
            \min_x \{ \bar{\alpha}^\T x \suchthat x \in P^t \} 
            = \min_p \{ \bar{\alpha}^\T p \suchthat p \in \pointset^t \}.
        $
    Define $\bar{\beta}_t \defeq \bar{\alpha}^\T p$.
    Note that the rays in $\rayset^t$ need not be considered, as the optimization problem must be bounded since $\bar{\alpha}^\T x \ge \bar{\beta}_t$ is valid for all $x \in P^t$.
    The point $p^t$ has a unique basis, so the basis cone $C^t$ is defined by the (precisely) $n$ constraints of $P^t$ that are tight at $p^t$.
    Optimality of $p^t$ implies all reduced costs are nonnegative.
    It follows that $\bar{\alpha}^\T r \ge 0$ every ray $r \in C^t$.
    Since $\bar{\alpha}^\T p^t = \bar{\beta}_t \ge \bar{\beta}$, 
    the inequality $\bar{\alpha}^\T x \ge \bar{\beta}$ is valid for $C^t$.
\end{proof}

Therefore, we can invoke \cref{lem:weaker-certificate,lem:A^t-invertible} to find the Farkas certificate for this case.
Note that, when the given point-ray collection only contains extreme points and rays, 
the rays of $C^t$ for any basis cone of the simple polyhedron $P^t$ can be computed as the rays $\rayset^t$,
along with the directions $p - p^t$ for every point $p \in \pointset^t$ that is adjacent (one pivot away) from $p^t$.

\subsection{Relaxations with Primal Degeneracy}
\label{sec:general-vpcs}

Up to now, we have made the convenient assumption that the relaxation $P^t$ is a simple polyhedron.
More generally, there always exists a basis cone $C^t$, such that a cut valid for $P^t$ is valid for $C^t$.
With \cref{ex:degeneracy}, we illustrate the complication if $\bar{\alpha}^\T x \ge \bar{\beta}$ is supporting at a primal degenerate point of $P^t$:
a basis for that point needs to be chosen carefully, as the inequality may not be valid for some basis cones.
It can be computationally involved to find a valid basis in these situations, which prevents a direct application of our approach relying on simple polyhedra.
The purpose of this example is to highlight a crucial obstacle to a complete correspondence between PRLP and CGLP solutions, but we do not further investigate the nondegenerate case in this paper.

\begin{example}
\label{ex:degeneracy}

\begin{figure}[t]
  \centering
  \begin{subfigure}{.49\textwidth}
  \centering
      \begin{tikzpicture}[line join=round,line cap=round,x={(\xX cm, \xY cm)},y={(\yX cm, \yY cm)},z={(\zX cm,\zY cm)},>=stealth,scale=2.75] 
        \coordinate (p1) at (1/2,1/4,1);
        \coordinate (p2) at (4/9,13/18,35/36);
        
        \coordinate (p11) at (0,1/2,0);
        \coordinate (p12) at (0,1/2,3/4);
        
        \coordinate (p21) at (1,0,0);
        \coordinate (p22) at (1,1,0);
        \coordinate (p24) at (1,0,1/4);
        \coordinate (p23) at (1,1,0);
        
        \draw [fill opacity=0.5,fill=none, color=green!80!black, line width=1pt] (0,0,0) -- (0,1,0) -- (0,1,1) -- (0,0,1) -- cycle;
        
        \draw [polyhedron_edge] (p11) -- (p12) -- (p1) -- (p24) -- (p21) -- cycle;
        \draw [polyhedron_edge,dashed] (p11) --  (p22);
        \draw [polyhedron_edge] (p12) -- (p1) -- (p2) -- cycle;
        \draw [polyhedron_edge] (p24) -- (p1) -- (p2) -- (p23) -- cycle;
        \node [circle, inner sep=1.5pt, fill=none] at (0,1.35,.65) {\color{green!80!black} {\small $x_1 = 0$}};
        \node [circle, inner sep=1.5pt, fill=none] at (1,1.35,0.65) {\color{green!80!black} {\small \clap{$x_1 = 1$}}};
        \draw [fill opacity=0.5,fill=none, color=green!80!black, line width=1pt] (1,0,0) -- (1,1,0) -- (1,1,1) -- (1,0,1) -- cycle;
        \draw [polyhedron_edge] (p21) -- (p22) -- (p23) -- (p24) -- cycle;
        
        \node [circle, inner sep=1.5pt, fill=none] at (1/4,1,3/8) {\footnotesize \eqref{ex:c1}};
        \node [circle, inner sep=1.5pt, fill=none] at (3/16,1/2,1/2) {\footnotesize \eqref{ex:c2}};
        \node [circle, inner sep=1.5pt, fill=none] at (3/16,1/2,2/2) {\footnotesize \eqref{ex:c3}};
      \end{tikzpicture}
  \end{subfigure}
  \begin{subfigure}{.49\textwidth}
  \centering
      \begin{tikzpicture}[line join=round,line cap=round,x={(\xX cm, \xY cm)},y={(\yX cm, \yY cm)},z={(\zX cm,\zY cm)},>=stealth,scale=2.75]
        \draw [fill opacity=0.5,fill=none, color=green!80!black, line width=1pt] (0,0,0) -- (0,1,0) -- (0,1,1) -- (0,0,1) -- cycle;
        
        \draw [polyhedron_edge,dotted] (p12) -- (p1) -- (p2) -- cycle;
        \draw [polyhedron_edge,dotted] (p1) -- (p24);
        \draw [polyhedron_edge,dotted] (p2) -- (p23);
        \draw [polyhedron_edge,dashed] (p11) -- (p22);

        \draw [polyhedron_edge] (p12) -- (p11) -- (p21);
        \draw [polyhedron_edge] (p21) -- (p22);
        
        \node [circle, inner sep=1.5pt, fill=none] at (0,1.35,.65) {\color{green!80!black} {\small $x_1 = 0$}};
        \node [circle, inner sep=1.5pt, fill=none] at (1,1.35,0.65) {\color{green!80!black} {\small \clap{$x_1 = 1$}}};
        \draw [fill opacity=0.5,fill=none, color=green!80!black, line width=1pt] (1,0,0) -- (1,1,0) -- (1,1,1) -- (1,0,1) -- cycle;
        \draw [polyhedron_edge] (p21) -- (p22) -- (p23) -- (p24) -- cycle;
        \draw [cut line] (p12) -- (p23) -- (p24) -- cycle;
		\node [point, label={[label distance=-3pt]135:\footnotesize \clap{$p^{1}$}}] at (p12) {};
		\node [point, label={[label distance=-3pt]180:\footnotesize {$p^{2}$}}] at (p24) {};
		\node [point, label={[label distance=-3pt]0:\footnotesize \rlap{\ $p^{3}$}}] at (p23) {};
      \end{tikzpicture}
      \end{subfigure}
    \caption{\cref{ex:degeneracy}: Disjunctive terms with primal degeneracy, despite a nondegenerate initial polyhedron.
    The VPC is the red wavy line in the second panel.
    }
    \label{correspondence:fig:degeneracy}
\end{figure}
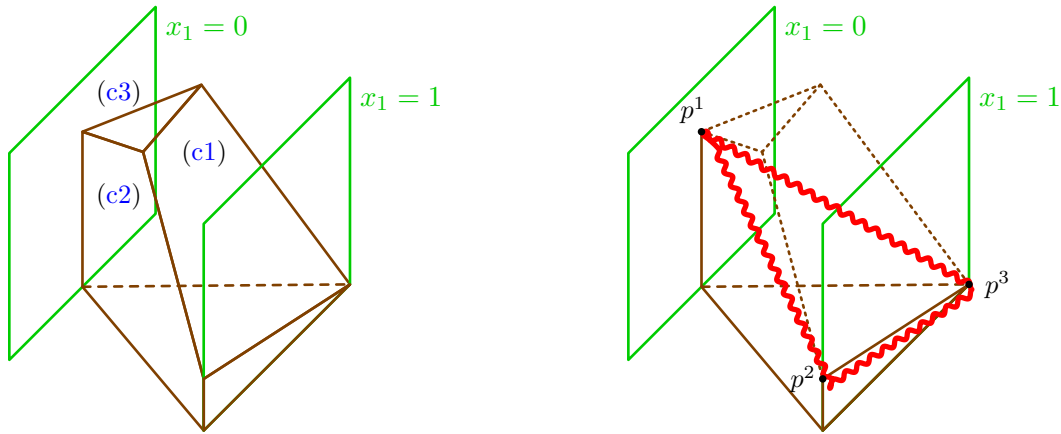

\cref{correspondence:fig:degeneracy} shows a polyhedron $P$, defined as the feasible solutions to %
  \begin{align}
    -(13/8) x_1 - (1/4) x_2 - x_3 &\ge -15/8 \tag{c1}\label{ex:c1}\\
    (1/2) x_1 {}+{} x_2 &\ge 1/2 \tag{c2}\label{ex:c2}\\
    (1/2) x_1 {}-{} x_3 &\ge -3/4 \tag{c3}\label{ex:c3}\\
    (1/2) x_1 {}-{} x_2 &\ge -1/2 \tag{c4}\label{ex:c4}\\
     x_2 &{}\ge{} 0. \tag{c5}\label{ex:c5}
  \end{align}

A valid cut from the disjunction $(-x_1 \ge 0) \vee (x_1 \ge 1)$ has coefficients
$\bar{\alpha}^\T = (-5/8,-1/4,-1)$ and constant $\bar{\beta} = -7/8$.
The cut, depicted in the right panel, is incident to point $p^1 = (0,1/2,3/4)$ on $P^1 \defeq \{x \in P \suchthat -x_1 \ge 0\}$.
This point is tight for four inequalities: three defining $P$ (constraints \cref{ex:c2,ex:c3,ex:c4}), and the disjunction-defining inequality $-x_1 \ge 0$.
Note that $P$ is simple, but $P^1$ is not.

To construct the cobasis $N^1$, such that the inequality is valid for the associated basis cone $C^1$,
we must select three linearly independent constraints among those that are tight at $p^1$.
One of the inequalities must be $-x_1 \ge 0$, as otherwise we have not imposed the disjunction at all (but we also know the cut is not valid for $P$).
It can be verified that the only valid choice for this example is $N^1$ containing the indices for \cref{ex:c3}, \cref{ex:c4}, and the disjunctive inequality $-x_1 \ge 0$.
\exqed
\end{example}

\section{Computational Experiments}
\label{sec:computation}

We implement monoidal strengthening for simple VPCs, building on the code used by \citet{BalKaz22+_vpc-arxiv} from \url{https://github.com/akazachk/vpc}.
Our goal for the computational study is to measure the effect of monoidal strengthening on the \emph{percent integrality gap closed} by VPCs, compared to unstrengthened VPCs and GMICs,
and evaluated across different disjunction sizes.

The code is run on HiPerGator, a shared cluster through Research Computing at the University of Florida.
The computational setup is nearly identical to the one described in \citet[Section~5 and Appendix~C]{BalKaz22+_vpc-arxiv}.
We select instances from the union of the MIPLIB~\cite{MIPLIB,MIPLIB3,MIPLIB2003,MIPLIB2010,MIPLIB2017}, CORAL~\cite{CORAL}, and NEOS
sets, restricted to those with at most 5,000 rows and columns and based on other criteria given in 
\cite[Appendix~C]{BalKaz22+_vpc-arxiv}.
This yields 332 instances suitable for gap closed comparisons.
However, we only report on 274 of these 332 instances, due to memory resource constraints on the cluster.
Despite this reduced dataset, we can identify recurring patterns in how monoidal strengthening affects instances.
Instances are presolved with Gurobi~\cite{Gurobi}, but cut generation is done via the \Cpp{} interface to COIN-OR~\cite{COIN-OR}, 
using Clp~\cite{Clp} for solving linear programs and Cbc~\cite{Cbc} for constructing disjunctions based on partial branch-and-bound trees.
We test six different disjunction sizes, stopping branching when the number of leaf nodes (disjunctive terms) is $2^\ell$ for $\ell \in [6]$.
Thus, we report results with monoidal strengthening of disjunctive cuts from up to 64-term disjunctions,
though only one disjunction is used at a time.
One GMIC is generated per fractional integer variable at an optimal solution to the linear programming relaxation,
and the number of GMICs is also used as the limit for the number of VPCs we generate for that instance per fixed choice of disjunction.
One round of cuts is used for both procedures.
GMICs are generated through \texttt{CglGMI}~\cite{Cgl},
while the VPC generation procedure is identical to that of \citet{BalKaz22+_vpc-arxiv},
with strengthening applied afterwards.

While \cref{lem:A^t-invertible} enables us to calculate the values of the Farkas multipliers via the rays of each relaxation $P^t$,
and these values are readily available based on how we built the PRLP, we do not avail of this connection.
Instead, we calculate $v^t = \bar{\alpha}^\T (A^t)^{-1}$.
This approach is still more direct than solving a feasibility version of \cref{CGLP-fixed-beta} with $\bar{\alpha}$ fixed.
We opt for numerical safety for this exploratory investigation, so we use the Eigen library~\cite{Eigen3} to recompute the inverse of $A^t$ rather than reading from the Cbc / Clp internal basis inverse for each term.

We report the average percent integrality gap closed by VPCs and GMICs in \cref{tab:depth}.
The first six data rows contain the results for each fixed disjunction size.
The penultimate data row, labeled ``Best'', uses the highest gap closed per instance across all disjunctions.
The last data row, labeled ``Wins'', reports the number of instances for which the ``Best'' gap closed is at least $10^{-3}$ higher than the gap closed by GMICs.
In the columns, we refer to
GMICs by ``G'',
unstrengthened VPCs by ``V'',
strengthened VPCs by ``V\textsuperscript{+}''.
The columns ``G+V'' and ``G+V\textsuperscript{+}'' refer to GMICs applied together with VPCs.
There are two sets of instances: ``All'' reports on all 274 instances, 
while ``\goodVPCSet'' reports on the 97 instances for which unstrengthened VPCs alone close at least 10\% of the integrality gap for the ``Best'' values.

In terms of overall gap closed, despite the monoidal strengthening procedure,
as reported by \citet{BalKaz22+_vpc-arxiv}, VPCs alone do not outperform GMICs for the ``All'' set,
but using VPCs and GMICs together provides around 40\% improvement in gap closed relative to GMICs alone.
While adding VPCs with GMICs might double the number of cuts,
one round of VPCs continues to close substantial more gap even after multiple rounds of solver-default cuts~\cite{BalKaz22+_vpc-arxiv}.
Hence, VPCs tighten the relaxation in different regions relative to GMICs.
This is also highlighted by the ``\goodVPCSet'' set, which are instances for which VPCs have strong performance;
for this set, GMICs are relatively weaker, with the best VPCs per instance (used alone) providing a 75\% improvement in average percent gap closed over GMICs alone.
We also see this in the ``Wins'' row: for the ``\goodVPCSet'' set, VPCs alone outperform GMICs for 73 of the 97 instances in the set.

{
\sisetup{
    table-alignment-mode = format,
    table-number-alignment = center,
    table-format = 2.2,
}
\begin{table}[t]
\renewcommand{\tabcolsep}{.45em}
\centering
\caption{
            Average percent gap closed by VPCs and GMICs according to the number of leaf nodes used to construct the partial branch-and-bound tree.
            ``Best'' refers to the maximum gap closed per instance across all partial tree sizes.
        }
\label{tab:depth}
\begin{tabular}{@{}
            l
            @{\hskip 1em}
            *{5}{S[table-auto-round]}
            *{5}{S[table-auto-round]}
        @{}}
\toprule
& \multicolumn{5}{c}{All}
& \multicolumn{5}{c}{\goodVPCSet}
\\
\cmidrule(r){2-6} \cmidrule(l){7-11}
{} 
& {G} & {V} & {V\textsuperscript{+}} & {G+V} & {G+V\textsuperscript{+}} 
& {G} & {V} & {V\textsuperscript{+}} & {G+V} & {G+V\textsuperscript{+}}
\\
\midrule

2 leaves  & 17.21305672 & 2.283558442 & 3.246264847 & 17.95356823 & 18.17716805 & 16.29044662 & 5.336366979 & 6.47426168  & 18.13170743 & 18.58662604 \\
4 leaves  & 17.21305672 & 3.34594685  & 3.722825872 & 18.36777648 & 18.54010478 & 16.29044662 & 7.80978766  & 8.345640969 & 19.13856557 & 19.48467066 \\
8 leaves  & 17.21305672 & 4.510602274 & 4.759439365 & 18.98300736 & 19.14755959 & 16.29044662 & 10.83618603 & 11.16369328 & 20.66217997 & 20.9051507  \\
16 leaves & 17.21305672 & 6.414229697 & 6.567021814 & 20.54136884 & 20.6742689  & 16.29044662 & 15.81121835 & 16.05036938 & 24.86070985 & 25.03628466 \\
32 leaves & 17.21305672 & 8.777608011 & 8.97137408  & 22.30956192 & 22.47582575 & 16.29044662 & 21.82063398 & 22.28386503 & 29.58939513 & 29.9739158  \\
64 leaves & 17.21305672 & 10.45577873 & 10.57470127 & 23.719387   & 23.83269303 & 16.29044662 & 25.58778496 & 25.84561465 & 32.89875982 & 33.13913263 \\
\midrule
Best      & 17.21305672 & 11.93077401 & 12.57326639 & 24.6668168  & 24.88933653 & 16.29044662 & 29.26134773 & 29.52557606 & 35.26854676 & 35.58975427
\\
Wins      &             & {103}	      & {104}	    & {185}       &	{190}       &             & {73}        &	{73}        & {94}        &	{94}
\\
\bottomrule
\end{tabular}
\end{table}
}

Next, we summarize observations about the effect of monoidal strengthening.
We start with the first data row, in which VPCs are derived from one split disjunction per instance.
For the set ``All'', monoidal strengthening affects the gap closed by VPCs for 87 instances and increases the average gap closed by VPCs by \textasciitilde 1\% from 2.28\% to 3.25\%, a 40\% relative improvement.
For the set ``\goodVPCSet'', the corresponding relative improvement is 20\%.

Although the two-term case is encouraging, and a similar relative improvement in gap closed would be substantial for larger disjunctions, 
this unfortunately does not materialize.
From \cref{tab:depth}, we see that as the disjunction size increases, the contribution of monoidal strengthening tends to further diminish, with an absolute improvement in gap closed of only 0.1\% for VPCs from a 64-term disjunction.
We will discuss a potential cause for this in the next section.

We now compare the columns ``G+V\textsuperscript{+}'' to ``G+V''.
On the set ``All'', even for split disjunctions, the effect of strengthening is minimal when VPCs are combined with GMICs, with strengthening only yielding an additional $0.23\%$ in percent gap closed, preserving around 23\% of the improvement between ``V\textsuperscript{+}'' and ``V''.
For larger disjunctions, while the absolute increase in gap closed by strengthened VPCs is small,
over 80\% of that improvement is preserved when adding GMICs together with VPCs.

A closer examination of the results supports the hypothesis
that monoidal strengthening is a key factor enabling GMICs to close more gap than VPCs.
We sort the instances by the increase in gap closed by strengthened VPCs compared to unstrengthened ones, using the best gap closed across all disjunction sizes, per column.
\cref{tab:top5} shows the top ten instances, sorted by the last column, which calculates the difference between ``V\textsuperscript{+}'' and ``V''.
The table includes the instance \instance{10teams} discussed earlier, as well as six other instances for which unstrengthened VPCs close at most 5\% of the gap.
We see that monoidal strengthening of VPCs bridges a large portion of the difference with GMICs for these instances.
For \instance{neos-1281048}, the situation is reversed: 121 GMICs close no gap while 29 unstrengthened VPCs close 17\% of the gap, which is further improved to 29\% after strengthening.
From this table, we also observe the phenomenon that the value in column ``G+V'' is typically either entirely due to GMICs or to VPCs, but which cuts are more important varies by instance.
The situation remains similar for the column ``G+V\textsuperscript{+}'', though now we find several cases (\instance{f2gap401600}, \instance{p0548}, \instance{mkc}) in which the two cut families add to each other.

{
\sisetup{
    table-alignment-mode = format,
    table-number-alignment = center,
    table-format = 3.2,
}
\begin{table}[t]
\renewcommand{\tabcolsep}{.45em}
\centering
\caption{
            Percent gap closed for instances where strengthening VPCs works best.
        }
\label{tab:top5}
\begin{tabular}{@{}
            l
            *{1}{S[table-auto-round,table-format=3.2]}
            *{1}{S[table-auto-round,table-format=2.2]}
            *{4}{S[table-auto-round,table-format=3.2]}
        @{}}
\toprule
{Instance} & {G} & {V} & {V\textsuperscript{+}} & {G+V} & {G+V\textsuperscript{+}} & {V\textsuperscript{+}$-$V}
\\
\midrule
\instance{10teams}            & 100       & 0         & 100       & 100       & 100       & 100       \\
\instance{neos-1281048}       & 0         & 17.08645  & 29.356848 & 17.08645  & 29.356848 & 12.270398 \\
\instance{neos-1599274}       & 34.64875  & 0         & 11.19238  & 34.64875  & 34.64875  & 11.19238  \\
\instance{f2gap401600}        & 62.966944 & 2.532368  & 11.335229 & 63.309065 & 71.772235 & 8.802861  \\
\instance{prod2}              & 2.30907   & 27.603289 & 35.897319 & 27.63088  & 35.905017 & 8.29403   \\
\instance{neos-942830}        & 6.25      & 0         & 6.25      & 6.25      & 6.25      & 6.25      \\
\instance{p0548}              & 48.621041 & 3.278173  & 9.030664  & 49.028759 & 55.105255 & 5.752491  \\
\instance{mkc}                & 6.079731  & 2.600544  & 6.56202   & 6.347367  & 9.61108   & 3.961476  \\
\instance{f2gap201600}        & 60.271212 & 8.575327  & 12.131208 & 60.271212 & 60.271212 & 3.555881  \\
\instance{neos-4333596-skien} & 20.835984 & 7.048222  & 9.830355  & 20.835984 & 20.852802 & 2.782133 
\\
\bottomrule
\end{tabular}
\end{table}
}

While running time is not our focus, and the shared computing environment makes wall clock times unreliable, 
\cref{tab:running-time} provides 
the average number of seconds for a single run of each instance, including generating then strengthening VPCs.
On average, cut generation takes, in total, from less than a second for two-term disjunctions to 50 seconds for 16-term disjunctions, 150 seconds for 32-term disjunctions, and nearly 9 minutes for 64-term disjunctions.
The time per cut, on average, is less than 0.1 seconds for two-term disjunctions, ranging up to 9 seconds for 32 terms and over 30 seconds for 64 terms.

{
\renewcommand{\textfloatsep}{-1em}
\renewcommand{\arraystretch}{0.65}
\sisetup{
    table-alignment-mode = format,
    table-number-alignment = center,
    table-format = 3.2,
}
\begin{table}[b]
\renewcommand{\tabcolsep}{.45em}
\centering
\caption{
            Average time (seconds) to generate the cuts in column {V\textsuperscript{+}} of \cref{tab:depth}.
        }
\label{tab:running-time}
\begin{tabular}{@{}
            ll
            *{2}{S[table-auto-round,table-format=1.2]}
            *{2}{S[table-auto-round,table-format=2.2]}
            *{3}{S[table-auto-round,table-format=3.2]}
        @{}}
\toprule
{Statistic} & {Set} & {2 leaves} & {4 leaves} & {8 leaves} & {16 leaves} & {32 leaves} & {64 leaves}
\\
\midrule
\multirow{2}*{Cut time (s)} & {All} & 0.755428571	& 6.38952381 & 15.3302381 & 49.89590476 & 149.835381 & 525.7790476
\\
& {\goodVPCSet} & 0.917010309	& 9.314536082   & 21.05742268   & 130.4549485   &	273.5126804   &	521.9860825
\\
\midrule
\multirow{2}*{Time/cut (s)} & {All} & 0.078746636	& 0.394118765   & 0.973595199	  & 2.646753356	  & 8.996324887	  & 30.53973419
\\
& {\goodVPCSet} & 0.066793661   &	0.346145955   &	0.785482158	  & 2.460707448	& 7.752531074	& 20.19377575
\\
\bottomrule
\end{tabular}
\end{table}
}

\section{Choosing a Relaxation Amenable to Strengthening}
\label{sec:example:large-disjunction}

In this section, we examine a potential cause of the diminishing effect of monoidal strengthening with larger disjunctions.
From \cref{thm:monoidal-strengthening-structural-space}, given an initial cut $\alpha^\T x \ge \beta$, 
we can strengthen coefficient $\alpha_k$, $k \in \intvars$, to
    \begin{equation*}
      \widetilde{\alpha}_k 
        = \alpha_k + \inf_{\monoid \in \M} \max_{t \in \disjTermsIndexSet} \left\{ -\hat{u}^t_k + u^t_0 \Deltat \monoid_t \right\},
    \end{equation*}
where
    $\hat{u}^t_k = \alpha_k - (u^t \mxcol{A}{k} + u^t_0 \mxcol{D^t}{k})$
is the slack on the CGLP constraint
    $
        \alpha_k \ge u^t \mxcol{A}{k} + u^t_0 \mxcol{D^t}{k}.
    $
Equivalently, $\hat{u}^t_k$ is the Farkas multiplier for the nonnegativity constraint $x_k \ge 0$.
The next lemma restates the (known) reason that a \emph{nonbasic} integral variable $k$ is required for monoidal strengthening.

\begin{lemma}
    If $\hat{u}^t_k = 0$, then $\widetilde{\alpha}_k = \alpha_k$.
\end{lemma}
\begin{proof}
    In this case,
    \(
        \widetilde{\alpha}_k 
            = \alpha_k + \inf_{\monoid \in \M} \max_{t \in \disjTermsIndexSet} \left\{ u^t_0 \Deltat \monoid_t \right\}.
    \)
    Since $\sum_{t \in \disjTermsIndexSet} \monoid_t \ge 0$ for every $\monoid \in \monoidSet$,
    and $u^t_0 \Deltat \ge 0$, the optimal solution is $\monoid = \vec{0}$.
\end{proof}

In the correspondence in \cref{sec:correspondence}, we ultimately find a point $p^t \in P^t$ such that $\bar{\alpha}^\T p^t = \bar{\beta}_t$,
where $\bar{\beta}_t = \min_p \{ \bar{\alpha}^\T p \suchthat p \in \pointset^t \}$.
We then compute a basis cone at $p^t$ for which the cut is valid and use this (translated) cone to compute the values of the Farkas certificate.
However, by complementary slackness, if $p^t_k > 0$, then necessarily $\hat{u}^t_k = 0$.

Although at first this appears simultaneously unfortunate and unavoidable,
there are two potential remedies.
First, there may be dual degeneracy in the choice of $p^t$: each such point can lead to a different Farkas certificate and therefore a different strengthening.
Second, as observed by \citet{BalQua12}, ``sometimes weaking a disjunction helps the strengthening''.
Though in that context, the weakening involves adding terms to the disjunction, the sentiment applies to our setting as well:
if $\bar{\beta}_t > \bar{\beta}$, then one can seek a different, potentially infeasible, basis of $Q^t$ in which more integer variables are nonbasic and $\bar{\alpha}^\T x \ge \bar{\beta}$ is still valid for the associated basis cone.

The computational results support the above intuition. %
When VPCs are generated from a split disjunction, on average, around 95\% of the generated cuts per instance have any coefficient strengthened with the monoidal technique.
This decreases to 85\% for 64-term disjunctions.
Furthermore, on average among VPCs to which strengthening has been applied, 20\% of the cut coefficients are strengthened for split disjunctions, while this value steadily decreases as disjunction size increases,
so among the analogous VPCs from 64-term disjunctions, only 10\% of the coefficients are strengthened.

\section{Conclusion}
\label{sec:conclusion}

We show that strengthening cuts from general disjunctions is possible without explicitly solving a higher-dimensional CGLP,
and that this strengthening can have a high impact for certain instances.
However, several challenges are also highlighted for future work.
First, the strengthening does not work well on average for larger disjunctions.
While we propose a viable explanation and remedy, it is computationally demanding and requires development.
Second, the optimal monoidal strengthening involves solving an integer program per cut; this is a relatively small and easy problem, but it nonetheless can be slow
for larger disjunctions, as suggested by \cref{tab:running-time}, which includes strengthening time.
One can reduce this load by selectively strengthening only the most promising cuts, identified by theoretical properties or good heuristics,
or to forego optimality in the strengthened cut coefficients.
Our computational results indicate that VPCs and GMICs seem to have complementary affects; understanding this better is an opportunity to more widely adopt disjunctive cuts.

\ifspringer
    \bibliographystyle{splncs04}
\else
    \bibliographystyle{plainnat}
\fi
\bibliography{2023-02-28_akazachk}

\begin{thebibliography}{50}
\providecommand{\natexlab}[1]{#1}
\providecommand{\url}[1]{\texttt{#1}}
\expandafter\ifx\csname urlstyle\endcsname\relax
  \providecommand{\doi}[1]{doi: #1}\else
  \providecommand{\doi}{doi: \begingroup \urlstyle{rm}\Url}\fi

\bibitem[Cbc()]{Cbc}
{COIN-OR} {Branch and Cut}.
\newblock \url{https://github.com/coin-or/Cbc}.

\bibitem[Cgl()]{Cgl}
{COIN-OR} {Cut} {Generation} {Library}.
\newblock \url{https://github.com/coin-or/Cgl}.

\bibitem[Clp()]{Clp}
{COIN-OR} {Linear} {Programming}.
\newblock \url{https://github.com/coin-or/Clp}.

\bibitem[Achterberg et~al.(2006)Achterberg, Koch, and Martin]{MIPLIB2003}
T.~Achterberg, T.~Koch, and A.~Martin.
\newblock {MIPLIB} 2003.
\newblock \emph{Oper. Res. Lett.}, 34\penalty0 (4):\penalty0 361--372, 2006.

\bibitem[Andersen et~al.(2005)Andersen, Cornu{\'e}jols, and Li]{AndCorLi05}
Kent Andersen, G{\'e}rard Cornu{\'e}jols, and Yanjun Li.
\newblock Split closure and intersection cuts.
\newblock \emph{Math. Program.}, 102\penalty0 (3, Ser. A):\penalty0 457--493,
  2005.

\bibitem[Balas and Qualizza(2012)]{BalQua12}
E.~Balas and A.~Qualizza.
\newblock Monoidal cut strengthening revisited.
\newblock \emph{Discrete Optim.}, 9\penalty0 (1):\penalty0 40--49, 2012.

\bibitem[Balas(1971)]{Balas71}
Egon Balas.
\newblock Intersection cuts---a new type of cutting planes for integer
  programming.
\newblock \emph{Oper. Res.}, 19\penalty0 (1):\penalty0 19--39, 1971.

\bibitem[Balas(1979)]{Balas79}
Egon Balas.
\newblock Disjunctive programming.
\newblock \emph{Ann. Discrete Math.}, 5:\penalty0 3--51, 1979.

\bibitem[Balas and Jeroslow(1980)]{BalJer80}
Egon Balas and Robert~G. Jeroslow.
\newblock Strengthening cuts for mixed integer programs.
\newblock \emph{European J. Oper. Res.}, 4\penalty0 (4):\penalty0 224--234,
  1980.

\bibitem[Balas and Kazachkov(2022)]{BalKaz22+_vpc-arxiv}
Egon Balas and Aleksandr~M. Kazachkov.
\newblock {$\mathcal{V}$}-polyhedral disjunctive cuts, 2022.
\newblock URL \url{https://arxiv.org/abs/2207.13619}.

\bibitem[Balas and Kis(2016)]{BalKis16}
Egon Balas and Tam{\'a}s Kis.
\newblock On the relationship between standard intersection cuts,
  lift-and-project cuts and generalized intersection cuts.
\newblock \emph{Math. Program.}, pages 1--30, 2016.

\bibitem[Balas and Perregaard(2003)]{BalPer03}
Egon Balas and Michael Perregaard.
\newblock A precise correspondence between lift-and-project cuts, simple
  disjunctive cuts, and mixed integer {Gomory} cuts for {$0$}-{$1$}
  programming.
\newblock \emph{Math. Program.}, 94\penalty0 (2-3, Ser. B):\penalty0 221--245,
  2003.
\newblock The Aussois 2000 Workshop in Combinatorial Optimization.

\bibitem[Balas and Qualizza(2013)]{BalQua13}
Egon Balas and Andrea Qualizza.
\newblock Intersection cuts from multiple rows: a disjunctive programming
  approach.
\newblock \emph{EURO J. Computat. Optim.}, 1\penalty0 (1):\penalty0 3--49,
  2013.

\bibitem[Balas et~al.(1993)Balas, Ceria, and Cornu{\'e}jols]{BalCerCor93}
Egon Balas, Sebasti{\'a}n Ceria, and G{\'e}rard Cornu{\'e}jols.
\newblock A lift-and-project cutting plane algorithm for mixed {$0$}-{$1$}
  programs.
\newblock \emph{Math. Program.}, 58\penalty0 (3, Ser. A):\penalty0 295--324,
  1993.

\bibitem[Balas et~al.(1996)Balas, Ceria, and Cornu{\'e}jols]{BalCerCor96}
Egon Balas, Sebasti{\'a}n Ceria, and G{\'e}rard Cornu{\'e}jols.
\newblock Mixed {$0$}-{$1$} programming by lift-and-project in a branch-and-cut
  framework.
\newblock \emph{Man. Sci.}, 42\penalty0 (9):\penalty0 1229--1246, 1996.

\bibitem[Basu et~al.(2011)Basu, Bonami, Cornu{\'e}jols, and
  Margot]{BasBonCorMar11a}
Amitabh Basu, Pierre Bonami, G{\'e}rard Cornu{\'e}jols, and Fran{\c{c}}ois
  Margot.
\newblock Experiments with two-row cuts from degenerate tableaux.
\newblock \emph{INFORMS J. Comput.}, 23\penalty0 (4):\penalty0 578--590, 2011.

\bibitem[Bixby et~al.(1992)Bixby, Boyd, and Indovina]{MIPLIB}
R.~E. Bixby, E.~A. Boyd, and R.~R. Indovina.
\newblock {MIPLIB}: A test set of mixed integer programming problems.
\newblock \emph{SIAM News}, 25:\penalty0 16, 1992.

\bibitem[Bixby et~al.(1998)Bixby, Ceria, McZeal, and Savelsbergh]{MIPLIB3}
R.~E. Bixby, S.~Ceria, C.~M. McZeal, and M.~W.~P Savelsbergh.
\newblock An updated mixed integer programming library: {MIPLIB} 3.0.
\newblock \emph{Optima}, 58:\penalty0 12--15, 6 1998.

\bibitem[Bonami(2012)]{Bonami12}
Pierre Bonami.
\newblock On optimizing over lift-and-project closures.
\newblock \emph{Math. Program. Comput.}, 4\penalty0 (2):\penalty0 151--179,
  2012.

\bibitem[Bonami et~al.(2013)Bonami, Conforti, Cornu\'ejols, Molinaro, and
  Zambelli]{BonConCorMolZam13}
Pierre Bonami, Michele Conforti, G\'erard Cornu\'ejols, Marco Molinaro, and
  Giacomo Zambelli.
\newblock Cutting planes from two-term disjunctions.
\newblock \emph{Oper. Res. Lett.}, 41\penalty0 (5):\penalty0 442--444, 2013.

\bibitem[Conforti et~al.(2014)Conforti, Cornu{\'e}jols, and
  Zambelli]{ConCorZam14}
Michele Conforti, G{\'e}rard Cornu{\'e}jols, and Giacomo Zambelli.
\newblock \emph{Integer Programming}, volume 271 of \emph{Graduate Texts in
  Mathematics}.
\newblock Springer, Cham, 2014.

\bibitem[{CORAL}()]{CORAL}
{CORAL}.
\newblock {Computational Optimization Research at Lehigh}. {MIP} instances.
\newblock \url{coral.ise.lehigh.edu/data-sets/mixed-integer-instances/}, 2020.
\newblock Accessed September 2020.

\bibitem[Cornu{\'e}jols and Li(2001)]{CorLi01}
G{\'e}rard Cornu{\'e}jols and Yanjun Li.
\newblock Elementary closures for integer programs.
\newblock \emph{Oper. Res. Lett.}, 28\penalty0 (1):\penalty0 1--8, 2001.

\bibitem[Dey and Wolsey(2010)]{DeyWol10}
Santanu~S. Dey and Laurence~A. Wolsey.
\newblock Two row mixed-integer cuts via lifting.
\newblock \emph{Math. Program.}, 124\penalty0 (1-2, Ser. B):\penalty0 143--174,
  2010.

\bibitem[Dey et~al.(2014)Dey, Lodi, Tramontani, and Wolsey]{DeyLodTraWol14}
Santanu~S. Dey, Andrea Lodi, Andrea Tramontani, and Laurence~A. Wolsey.
\newblock On the practical strength of two-row tableau cuts.
\newblock \emph{INFORMS J. Comput.}, 26\penalty0 (2):\penalty0 222--237, 2014.

\bibitem[Espinoza(2010)]{Espinoza10}
Daniel~G. Espinoza.
\newblock Computing with multi-row {Gomory} cuts.
\newblock \emph{Oper. Res. Lett.}, 38\penalty0 (2):\penalty0 115--120, 2010.

\bibitem[Farkas(1902)]{Farkas02}
Julius Farkas.
\newblock Theorie der einfachen {U}ngleichungen.
\newblock \emph{J. Reine Angew. Math.}, 124:\penalty0 1--27, 1902.

\bibitem[Fischer and Pfetsch(2017)]{FisPfe17}
Tobias Fischer and Marc~E. Pfetsch.
\newblock Monoidal cut strengthening and generalized mixed-integer rounding for
  disjunctions and complementarity constraints.
\newblock \emph{Oper. Res. Lett.}, 45\penalty0 (6):\penalty0 556--560, 2017.

\bibitem[Fischetti et~al.(2011)Fischetti, Lodi, and Tramontani]{FisLodTra11}
Matteo Fischetti, Andrea Lodi, and Andrea Tramontani.
\newblock On the separation of disjunctive cuts.
\newblock \emph{Math. Program.}, 128\penalty0 (1-2, Ser. A):\penalty0 205--230,
  2011.

\bibitem[Fukasawa et~al.(2019)Fukasawa, Poirrier, and Xavier]{FukPoiXav19}
Ricardo Fukasawa, Laurent Poirrier, and {\'A}linson~S. Xavier.
\newblock The (not so) trivial lifting in two dimensions.
\newblock \emph{Math. Program. Comp.}, 11\penalty0 (2):\penalty0 211--235,
  2019.

\bibitem[Gleixner et~al.(2021)Gleixner, Hendel, Gamrath, Achterberg, Bastubbe,
  Berthold, et~al.]{MIPLIB2017}
A.~Gleixner, G.~Hendel, G.~Gamrath, T.~Achterberg, M.~Bastubbe, T.~Berthold,
  et~al.
\newblock {MIPLIB 2017: Data-Driven Compilation of the 6th Mixed-Integer
  Programming Library}.
\newblock \emph{Math. Prog. Comp.}, 2021.

\bibitem[Gomory and Johnson(1972)]{GomJoh72a}
Ralph~E. Gomory and Ellis~L. Johnson.
\newblock Some continuous functions related to corner polyhedra.
\newblock \emph{Math. Program.}, 3\penalty0 (1):\penalty0 23--85, 1972.

\bibitem[Guennebaud et~al.(2010)Guennebaud, Jacob, et~al.]{Eigen3}
Ga\"{e}l Guennebaud, Beno\^{i}t Jacob, et~al.
\newblock Eigen v3.
\newblock http://eigen.tuxfamily.org, 2010.

\bibitem[{Gurobi Optimization, LLC}(2022)]{Gurobi}
{Gurobi Optimization, LLC}.
\newblock {Gurobi Optimizer Reference Manual}, 2022.

\bibitem[Johnson(1974)]{Johnson74}
Ellis~L. Johnson.
\newblock On the group problem for mixed integer programming.
\newblock \emph{Math. Program. Stud.}, \penalty0 (2):\penalty0 137--179, 1974.

\bibitem[J{\'u}dice et~al.(2006)J{\'u}dice, Sherali, Ribeiro, and
  Faustino]{JudSheRibFau06_complementarity}
Joaquim~J. J{\'u}dice, Hanif~D. Sherali, Isabel~M. Ribeiro, and Ana~M.
  Faustino.
\newblock A complementarity-based partitioning and disjunctive cut algorithm
  for mathematical programming problems with equilibrium constraints.
\newblock \emph{J. Global Optim.}, 36\penalty0 (1):\penalty0 89--114, 2006.

\bibitem[Kazachkov(2018)]{Kazachkov18}
Aleksandr~M. Kazachkov.
\newblock \emph{Non-Recursive Cut Generation}.
\newblock PhD thesis, Carnegie Mellon University, 2018.

\bibitem[Kazachkov and Serrano()]{KazSer23}
Aleksandr~M. Kazachkov and Felipe Serrano.
\newblock Monoidal cut strengthening.
\newblock In Oleg Prokopyev and Panos~M. Pardalos, editors, \emph{Encyclopedia
  of Optimization}. Springer US, Boston, MA.
\newblock Under review.

\bibitem[K{\i}l{\i}n{\c{c}} et~al.(2014)K{\i}l{\i}n{\c{c}}, Linderoth, Luedtke,
  and Miller]{KilLinLueMil14}
Mustafa K{\i}l{\i}n{\c{c}}, Jeff Linderoth, James Luedtke, and Andrew Miller.
\newblock Strong-branching inequalities for convex mixed integer nonlinear
  programs.
\newblock \emph{Comput. Optim. Appl.}, 59\penalty0 (3):\penalty0 639--665,
  2014.

\bibitem[Kis(2014)]{Kis14}
Tam{\'a}s Kis.
\newblock Lift-and-project for general two-term disjunctions.
\newblock \emph{Discrete Optim.}, 12:\penalty0 98--114, 2014.

\bibitem[Koch et~al.(2011)Koch, Achterberg, Andersen, Bastert, Berthold, Bixby,
  et~al.]{MIPLIB2010}
T.~Koch, T.~Achterberg, E.~Andersen, O.~Bastert, T.~Berthold, R.~E. Bixby,
  et~al.
\newblock {MIPLIB} 2010: mixed integer programming library version 5.
\newblock \emph{Math. Program. Comput.}, 3\penalty0 (2):\penalty0 103--163,
  2011.

\bibitem[Kronqvist and
  Misener(2021)]{KroMis21_disj-cut-strengthening-convex-MINLP}
Jan Kronqvist and Ruth Misener.
\newblock A disjunctive cut strengthening technique for convex {MINLP}.
\newblock \emph{Optim. Eng.}, 22\penalty0 (3):\penalty0 1315--1345, 2021.

\bibitem[Lougee-Heimer(2003)]{COIN-OR}
Robin Lougee-Heimer.
\newblock {The Common Optimization INterface for Operations Research}:
  Promoting open-source software in the operations research community.
\newblock \emph{IBM Journal of Research and Development}, 47, 2003.

\bibitem[Louveaux et~al.(2015)Louveaux, Poirrier, and Salvagnin]{LouPoiSal15}
Quentin Louveaux, Laurent Poirrier, and Domenico Salvagnin.
\newblock The strength of multi-row models.
\newblock \emph{Math. Program. Comput.}, 7\penalty0 (2):\penalty0 113--148,
  2015.

\bibitem[Nemhauser and Wolsey(1988)]{NemWol88}
George~L. Nemhauser and Laurence~A. Wolsey.
\newblock \emph{Integer and combinatorial optimization}.
\newblock Wiley-Interscience Series in Discrete Mathematics and Optimization.
  John Wiley \& Sons, Inc., New York, 1988.

\bibitem[Nemhauser and Wolsey(1990)]{NemWol90}
George~L. Nemhauser and Laurence~A. Wolsey.
\newblock A recursive procedure to generate all cuts for {$0$}-{$1$} mixed
  integer programs.
\newblock \emph{Math. Program.}, 46\penalty0 (1):\penalty0 379--390, 1990.

\bibitem[Perregaard(2003)]{Perregaard03}
Michael Perregaard.
\newblock \emph{Generating Disjunctive Cuts for Mixed Integer Programs}.
\newblock PhD thesis, Carnegie Mellon University, 9 2003.

\bibitem[Perregaard and Balas(2001)]{PerBal01}
Michael Perregaard and Egon Balas.
\newblock Generating cuts from multiple-term disjunctions.
\newblock In \emph{Integer Programming and Combinatorial Optimization}, volume
  2081 of \emph{Lecture Notes in Comput. Sci.}, pages 348--360. Springer,
  Berlin, 2001.

\bibitem[Xavier et~al.(2021)Xavier, Fukasawa, and Poirrier]{XavFukPoi21}
\'{A}linson~S. Xavier, Ricardo Fukasawa, and Laurent Poirrier.
\newblock Multirow intersection cuts based on the infinity norm.
\newblock \emph{INFORMS J. Comput.}, 33\penalty0 (4):\penalty0 1624--1643,
  2021.

\bibitem[Ziegler(1995)]{Ziegler95}
G{\"u}nter~M. Ziegler.
\newblock \emph{Lectures on Polytopes}, volume 152 of \emph{Graduate Texts in
  Mathematics}.
\newblock Springer-Verlag, New York, 1995.

\end{thebibliography}

\newpage

\appendix

\end{document}